\documentclass[11pt,reqno]{amsart}
\usepackage{amsmath,mathtools}
\usepackage{graphicx}
\usepackage{epsfig,epsf,psfrag}
\usepackage{amsfonts}
\usepackage{setspace}
\usepackage{color}
\usepackage[font=small]{caption}
\usepackage[font=footnotesize]{subcaption}
\usepackage[top=1in, bottom=1in, left=1.2in, right=1.2in]{geometry}
\usepackage{afterpage}

\usepackage{isomath}

\newcommand{\vct}{\vectorsym}
\newcommand{\mtx}{\matrixsym}

\newcommand{\er}{\vectorsym{e}_r}
\newcommand{\et}{\vectorsym{e}_{\theta}}
\newcommand{\ez}{\vectorsym{e}_z}

\newcommand{\bn}{\vectorsym{n}}
\newcommand{\bt}{\vectorsym{t}}
\newcommand{\bx}{\vectorsym{x}}
\newcommand{\by}{\vectorsym{y}}

\newcommand{\bJ}{\vectorsym{J}}

\newcommand{\lp}{\left(}
\newcommand{\rp}{\right)}
\newcommand{\usc}{\boldsymbol u^{\rm sc}}
\newcommand{\bpsi}{\boldsymbol{\psi}}
\newcommand{\bta}{\boldsymbol{\tau}_1}
\newcommand{\btb}{\boldsymbol{\tau}_2}
\newcommand{\bti}{\boldsymbol{\tau}_i}
\newcommand{\mi}{\mathrm{i}}
\newcommand{\kp}{\kappa_{\mathfrak p}}
\newcommand{\ks}{\kappa_{\mathfrak s}}

\renewcommand{\phi}{\varphi}

\newtheorem{theorem}{Theorem}[section]

\newtheorem{lemma}[theorem]{Lemma}

\newtheorem{remark}[theorem]{Remark}

\title[]{A fast solver for elastic scattering from 
   axisymmetric objects by boundary integral equations} 

\author{Jun Lai}
\address{School of Mathematical Sciences, Zhejiang University,
	Hangzhou, Zhejiang 310027, China}
\email{laijun6@zju.edu.cn}

\author{Heping Dong}
\address{School of Mathematics, Jilin University, Changchun, Jilin 130012,
	China}
\email{dhp@jlu.edu.cn}

\thanks{The work of JL was partially supported by the Funds for Creative Research Groups of NSFC (No. 11621101) and NSFC grant No. 11871427. The work of HD was partially supported by NSFC grant No. 11801213 and the National Key Research and Development Program of China (grant No. 2020YFA0713602).}



\date{\today}

\begin{document}

\maketitle 
\begin{abstract}
  Fast and high-order accurate algorithms for three dimensional elastic scattering are of great importance when modeling  physical phenomena in mechanics, seismic imaging, and many other fields of applied science. In this paper, we develop a novel boundary integral formulation for the three dimensional elastic scattering based on the Helmholtz decomposition of elastic fields, which converts the Navier equation to a coupled system consisted of Helmholtz and Maxwell equations. An FFT-accelerated separation of variables  solver is proposed to efficiently invert boundary integral formulations of the coupled system for elastic scattering from axisymmetric rigid bodies. In particular, by combining the regularization properties of the singular boundary integral operators and the FFT-based fast evaluation of modal Green's functions, our numerical solver can rapidly solve the resulting integral equations with a high-order accuracy. Several numerical examples are provided to demonstrate the  efficiency and accuracy of the proposed algorithm, including geometries with corners at different wave number.
\end{abstract}


\section{Introduction}
\label{sec_intro}
Recently the phenomena of elastic scattering by obstacles have received much
attention due to the significant applications in various scientific areas such
as geological exploration, nondestructive testing, and medical diagnostics
\cite{ABG-15,HKS-IP13}.  An accurate and efficient numerical method
plays an important role in many of these applications. For instance, in the area of inverse elastic scattering, a large amount of forward simulations are needed in the iterative-type inversion algorithm \cite{DLL2019, DLL2020}. This
paper is concerned with the elastic scattering of a time-harmonic incident wave by a rigid obstacle embedded in a homogeneous and isotropic elastic medium in three dimensions. We propose a novel boundary integral formulation for the governing equation, that is, the Navier equation, and develop a highly accurate numerical method for solving the elastic scattering problem from axisymmetric bodies. 

Given the importance of elasticity, there are many mathematical and computational results available for the scattering problems of elastic waves \cite{AH-SIAP76,  L-SIAP12,  L2014, PV-JASA}.  Among these results, the
method of boundary integral equations offers an attractive approach for
solving  the obstacle scattering problems. Not only does it provide many analytic insights for the solution of elastic scattering, from the computational point of view, it also has the advantages that the discretization is only needed on the boundary of the domain and the radiation condition at infinity can be satisfied automatically \cite{PV-JASA}.  On the other hand, the Green's function of the elastic wave
equation is a second order tensor and the singularity is rather tedious to be separated in the computation of boundary integral equations, especially for the three dimensional case. Readers are referred to \cite{BXY-JCP17, BLR-JCP14, L2014, TC2007} and reference therein for some recent advances along this direction. To bypass this complexity, we use the Helmholtz decomposition to introduce two potential functions, one scalar and one vector, to split the displacement of the elastic wave fields into the compressional wave and the shear wave. The two potential functions, where the scalar one satisfies the Helmholtz equation and the vector one satisfies the Maxwell equation,
 are only coupled at the boundary of the
obstacle \cite{DLL2021}. Therefore, the boundary value problem of the Navier equation is converted equivalently into a coupled boundary value problem of Helmholtz and Maxwell equations for the potentials. Such a decomposition greatly reduces the complexity of computation for the elastic scattering problem, as many numerical techniques available for the boundary integral formulations for Helmholtz and Maxwell equations can be directly applied to the coupled problem, including the Fast Multiple Method (FMM)~\cite{GG2013, GB87}. Applications of this technique in two dimensional multi-particle scattering and inverse elastic scattering can be found in~\cite{DLL2019, DLL2020, LL2019}.

While in practice many elastic scattering problems require the
solution to Navier equations in arbitrary complex geometries,
it is also important to study the scattering problems in somewhat
simpler geometries, namely axisymmetric ones. Despite they simplify the persistent mathematical and computational difficulties of designing high-order methods in general geometries (e.g. quadrature design, mesh generation, etc), many practical problems, such as submarine detection, can be approximately modeled as the elastic scattering of axisymmetric objects. They can also provide a robust test-bed for the same integral equations which are used in general geometries and serve as building blocks for more complex objects, including composite materials made by axisymmetric particles \cite{GG2013,  Hao2015}.  In fact, the problem of computing scattered waves in axisymmetric geometries has a very rich history in both the acoustic and electromagnetic wave scattering communities~\cite{Geng1999Wide,gustafsson2010,Yu2008Closed},
and recently several groups have built specialized high-order solvers
for particular applications in plasma physics~\cite{oneil2018}, resonance
calculations~\cite{HK15, Helsing2016}, and inverse obstacle scattering~\cite{BL2020}. However, we are not aware of any efficient algorithm for the elastic scattering from axisymmetric objects.

In this work, based on the Helmholtz decomposition, the elastic scattered field decomposes into the gradient of a Helmholtz potential and the curl of a Maxwell potential.  We represent them by the Helmholtz and Maxwell single layer potentials, respectively. (The definition of Maxwell single layer potential will be given in Section \ref{sec_form}.) By using the boundary condition, we obtain a coupled integral equation system that consists of four singular boundary operators. It will be shown that two of them are only weakly singular after extracting the jumping terms, while the other two have Cauchy-type singularities. In order to design a high order discretization algorithm, we investigate the regularization properties of the two Cauchy-type singular operators, which will be reduced to weakly singular operators only,  and then apply the integral formulation to the scattering problem of axisymmetric objects. To evaluate the singular operators efficiently, we use a fast Fourier transform (FFT) based scheme to accelerate the evaluation of so-called modal Green's functions, which has been successfully applied in a variety of boundary value problems including  Laplace equations, Helmholtz equations~\cite{YHM2012}, and recently, Maxwell's equations~\cite{HK15, LAI2019}. Numerical results show that at the low frequency case, our algorithm can finish the computation in a few seconds on a laptop and achieve 10 digits accuracy compared with the reference solution. For the high frequency case, we are also able to get more than 6 digits accuracy in a reasonable amount of time. In addition, high order accuracy is also obtained for axisymmetric bodies that contain corners by using an adaptive discretization scheme. In summary,  our work in this paper has three main features: (1) a novel integral formulation for the elastic scattering problem based on Helmholtz decomposition, (2) new regularization properties for the singular integral operators resulted from the coupled boundary integral system, (3) a fast and high order discretization scheme for the elastic scattering of axisymmetric objects based on FFT.

We also make a few remarks on the integral formulation we choose. Analogous to  the single, double and combined layer potential representations for the Helmholtz equations~\cite{Kress2010}, one can also construct three variants for the integral representation of Navier equations. The integral formulation we adopt does not admit a unique solution in certain cases, as explained in Section~\ref{sec_form}. However, it does not undermine the value of the proposed algorithm, as the extension to the other two formulations would be straightforward.

An outline of the paper is the following: Section~\ref{probform}
introduces the problem of Navier equations for the time-harmonic elastic scattering in isotropic media. Section ~\ref{sec_form}
formulates the boundary integral equation system and shows the regularization properties of the singular operators.  Section~\ref{sec_fourier} gives the transformation of the integral equation along an axisymmetric surface into a sequence of decoupled integral equations along a curve in two dimensions based on the
Fourier transform in the azimuthal direction.  Ideas for the fast
kernel evaluation by FFT are given in Section~\ref{kerneleval}. It also discusses the high-order discretization of the sequence of line integrals using adaptive panels and generalized Gaussian quadratures for the associated weakly singular integral operators. Numerical examples are given in Section~\ref{sec_numeri}, including scattering results in both smooth and non-smooth geometries. Section~\ref{sec_con} concludes the discussion with some future work.

\section{Problem Formulation}
\label{probform}
Consider a three-dimensional rigid obstacle $\Omega$, given as a
simply connected, bounded domain in $\mathbb R^3$ with boundary $\Gamma$. To ease the discussion, we assume the boundary $\Gamma$ is smooth, although it will be numerically extended to non-smooth geometries.
Denote by $\boldsymbol{\tau}_1, \boldsymbol{\tau}_2$ and $\bn$ the unit tangential and exterior
normal vectors on $\Gamma$, respectively. The exterior domain $\mathbb{R}^3\setminus \overline{\Omega}$ is
assumed to be filled with a homogeneous and isotropic elastic medium with a unit
mass density.

Let the obstacle be illuminated by a time-harmonic compressional plane wave
$\boldsymbol{u}^{\rm inc}(\bx)= (\boldsymbol{d}\cdot \boldsymbol{p})\boldsymbol{d}\mathrm{e}^{\mathrm{i} \kappa_{\mathfrak p}\boldsymbol{d}\cdot \bx}$ or shear plane wave $\boldsymbol{u}^{\rm inc}(\bx)=(\boldsymbol{d}\times \boldsymbol{p})\times \boldsymbol{d} 
\mathrm{e}^{\mathrm{i}\kappa_{\mathfrak s} \boldsymbol{d}\cdot \bx}$,
where $\boldsymbol{d}$ and $\boldsymbol{p}$ are
the propagation and  polarization vectors, respectively, both of which are vectors on a unit sphere,  and
\[
\kappa_{\mathfrak p}=\frac{\omega}{\sqrt{\lambda+2\mu}},\quad
\kappa_{\mathfrak s}=\frac{\omega}{\sqrt{\mu}},
\]
are the compressional wavenumber and shear wavenumber, respectively. Here $\omega>0$ is the angular frequency and $\lambda, \mu$ are the
Lam\'{e} constants satisfying $\mu>0, \lambda+\mu>0$. The displacement of the total elastic field $\boldsymbol{u}$ satisfies the Navier
equation
\[
\mu\Delta\boldsymbol{u}+(\lambda+\mu)\nabla\nabla\cdot\boldsymbol{u}
+\omega^2\boldsymbol{u}=0\quad {\rm in}\, \mathbb{R}^3\setminus \overline{\Omega}. 	
\]
 Since the obstacle is
assumed to be rigid,  it implies $\boldsymbol u$ satisfies the homogeneous
boundary condition 
\[
\boldsymbol{u}=0\quad {\rm on}~\Gamma.
\]
The total field $\boldsymbol u$ consists of the incident field $\boldsymbol
u^{\rm inc}$ and the scattered field $\usc$, i.e., 
\[
\boldsymbol u=\boldsymbol u^{\rm inc}+\usc. 
\]
It is easy to verify that the incoming field $\boldsymbol u^{\rm inc}$ 
satisfies the Navier equation. Therefore the scattered field $\usc$ satisfies the boundary value problem
\begin{equation}\label{scatteredfield}
\begin{cases}
\mu\Delta\usc+(\lambda+\mu)\nabla\nabla\cdot\usc
+\omega^2\usc=0\quad &{\rm in}~
\mathbb{R}^3\setminus\overline{\Omega},\\
\usc=-\boldsymbol{u}^{\rm inc}\quad &{\rm on}~\Gamma.
\end{cases}
\end{equation}
In order to guarantee the well-posedness of the corresponding
boundary value problem, the  field $\usc$ is required to satisfy
the Kupradze radiation condition at infinity, i.e.
\begin{eqnarray}\label{krad}
\lim_{r\to\infty}r(\partial_r\boldsymbol u_{\mathfrak p}-\mathrm{i}\kappa_{\mathfrak p}\boldsymbol u_{\mathfrak p})=0,\quad
\lim_{r\to\infty}r(\partial_r\boldsymbol u_{\mathfrak s}-\mathrm{i}\kappa_{\mathfrak s}\boldsymbol u_{\mathfrak s})=0,\quad r=|\bx|,
\end{eqnarray}
uniformly in all directions, where
\[
\boldsymbol u_{\mathfrak p}=-\frac{1}{\kappa_{\mathfrak p}^2}\nabla\nabla\cdot\usc,\quad \boldsymbol u_{\mathfrak s}=\frac{1}{\kappa_{\mathfrak s}^2}\nabla\times \nabla\times \usc
\]
are known as the compressional and shear wave components of $\usc$,
respectively. 

While there exists the free space Green's function for the Navier equation, we choose to apply the Helmholtz decomposition to formulate the integral equation for the boudnary value problem \eqref{scatteredfield}. As we mentioned in the introduction, such a decomposition will simplify the computation of elastic scattering. For any solution $\usc$ of the elastic wave equation
\eqref{scatteredfield}, the Helmholtz decomposition reads
\begin{equation}\label{HelmDeco}
\usc=\nabla\phi+\nabla\times \boldsymbol{\psi}, \mbox{ with }\nabla\cdot \boldsymbol{\psi}=0,
\end{equation}
where $\phi$ and $\bpsi$ are scalar and vector potentials, respectively. Combining
\eqref{scatteredfield} and \eqref{HelmDeco} yields the Helmholtz and Maxwell equations
\begin{align}\label{helmmax}
\begin{cases}
\Delta\phi+\kappa_{\mathfrak p}^{2}\phi=0, \\ 
\nabla\times \nabla\times \bpsi-\kappa_{\mathfrak s}^{2}\bpsi=0.
\end{cases}
\end{align}
Due to the Kupradze radiation condition \eqref{krad}, $\phi$ and $\bpsi$ are required to satisfy the Sommerfeld radiation
conditions
\begin{equation*}
\lim_{r\to\infty}r(\partial_r
\phi-\mathrm{i}\kappa_{\mathfrak p}\phi)=0, \quad
\lim_{r\to\infty}r(\partial_r
\bpsi-\mathrm{i}\kappa_{\mathfrak s}\bpsi)=0, \quad r=|\bx|,
\end{equation*}
while the second equality is understood in the sense of component-wise, and is equivalent to the Silver-M\"uller radiation condition in electromagnetic scattering~\cite{DR-shu2}.
The coupling between $\phi$ and $\bpsi$ follows from  the boundary condition on
$\Gamma$ that
\[
\usc=\nabla\phi+\nabla\times\bpsi=-\boldsymbol{u}^{\rm
	inc}.
\]
Taking the dot and cross products of the above equation with the normal vector $\bn$,
respectively, we obtain
\[
 \partial_\bn\phi+\bn\cdot\nabla\times\bpsi=f,\quad
\bn\times\phi+\bn\times \nabla \times\bpsi=\boldsymbol{g},
\]
where
\[
f=-\bn\cdot\boldsymbol{u}^{\rm inc},\quad
\boldsymbol{g}=-\bn\times\boldsymbol{u}^{\rm inc}.
\]
In summary, the potential functions $\phi$ and $\bpsi$ satisfy the coupled
boundary value problem
\begin{align}\label{HelmholtzDec}
\begin{cases}
\Delta\phi+\kappa_{\mathfrak p}^{2}\phi=0,\quad
\nabla\times \nabla\times \bpsi-\kappa_{\mathfrak s}^{2}\bpsi=0,\quad &{\rm in}\, \mathbb{R}^3\setminus \overline{\Omega}, \\
\partial_\bn\phi+\bn\cdot\nabla\times\bpsi=f,\quad
\bn\times\nabla\phi+\bn\times \nabla \times\bpsi=\boldsymbol{g},\quad &{\rm on} ~ \Gamma,\\
\displaystyle{\lim_{r\to\infty}r(\partial_{r}\phi-\mathrm{i}\kappa_{\mathfrak p}\phi)=0}, \quad
\displaystyle{\lim_{r\to\infty}r(\partial_{r}\bpsi-\mathrm{i}\kappa_{\mathfrak s}\bpsi)=0}, \quad &r=|\bx|.
\end{cases}
\end{align}
When $\kappa_{\mathfrak p} >0$ and $\kappa_{\mathfrak s}>0$, the uniqueness result of the coupled system \eqref{HelmholtzDec} can be obtained by standard argument~\cite{LL2019}. We assume this always holds in the subsequent discussion, and mainly focus on solving system \eqref{HelmholtzDec} by the integral equation method.

\section{Integral Equation Formulations}
\label{sec_form}

 Using potential theory, we can represent the potential functions $\phi$ and $\bpsi$ in  \eqref{helmmax} by
\begin{eqnarray}\label{singlerep}
\phi =S^{\kp}\sigma, \quad
\bpsi =\frac{1}{\ks^{2}}\nabla\times \nabla\times S^{\ks}\mathbf{J}, \mbox{ for } \bx\in \mathbb{R}^3\setminus \overline{\Omega},
\end{eqnarray}
where
\begin{eqnarray}
S^{\kp}\sigma(\bx) &=& \int_{\Gamma} G^{\kp}(\bx,\by)\sigma(\by)ds_{\by},\label{singler}  \\
S^{\ks} \bJ(\bx) &=& \int_{\Gamma} G^{\ks}(\bx,\by) \, \bJ(\by) ds_{\by},\label{singlej} 
\end{eqnarray}
and the function
$G^{\kappa}(\bx,\by)$ is the free space Green's function for the three
dimensional Helmholtz equation:
\begin{equation}\label{greenf}
G^{\kappa}(\bx,\by) = \frac{e^{i{\kappa}|\bx-\by|}}{4\pi|\bx-\by|}, \quad \kappa=\kp \mbox{ or }\ks, 
\end{equation}  $\sigma$ is the unknown charge density and $\bJ$ is the unknown \emph{surface current} (as a terminology used in electromagnetic scattering) on $\Gamma$. While the representation for $\phi$ is the well-known Helmholtz single layer potential, for convenience we call the representation for $\bpsi$ in \eqref{singlerep} as the \emph{Maxwell single layer potential}. Based on the identity
\begin{eqnarray}
\frac{1}{\ks^{2}}\nabla\times \nabla\times S^{\ks}\mathbf{J} =\frac{1}{\ks^{2}} \nabla\nabla\cdot S^{\ks}\mathbf{J} - \frac{1}{\ks^{2}}\Delta S^{\ks}\mathbf{J} =\frac{1}{\ks^{2}} \nabla\nabla\cdot S^{\ks}\mathbf{J} + S^{\ks}\mathbf{J}, 
\end{eqnarray}
the solution $\mathbf{u}^{sc}$ can be represented by 
\begin{eqnarray}\label{scatrep}
\mathbf{u}^{sc}=\nabla\phi + \nabla\times \bpsi = \nabla S^{\kp}\sigma + \nabla\times S^{\ks}\mathbf{J}.
\end{eqnarray}

To derive the integral equation, we need the limiting behavior of the integral operators, since when $\bx$ approaches the boundary $\Gamma$, the integrals $\phi$ and $\bpsi$ become singular.  Denote by $C^{0,\alpha}(\Gamma)$, $0<\alpha< 1$, the classic \textit{H\"older space}, which is a Banach space with the norm
\[
||\sigma||_{0,\alpha}:=\sup_{\bx\in \Gamma}|\sigma(\bx)|+\sup_{\bx,\by\in \Gamma, \bx\ne \by}\frac{|\sigma(\bx)-\sigma(\by)|}{|\bx-\by|^\alpha}.
\]
Let $C^{1,\alpha}(\Gamma)$ be the Banach space with the norm
\[
||\sigma||_{1,\alpha}:=\sup_{\bx\in \Gamma}|\sigma(\bx)|+\sup_{\bx\in \Gamma}||\nabla\sigma||_{0,\alpha}.
\]
Following the notation in \cite{DR-shu2}, we denote by $T^{0,\alpha}(\Gamma)$, $0<\alpha<1$, the space of all uniformly H\"older continuous tangential fields on $\Gamma$, and define $T_{d}^{0,\alpha}(\Gamma)$ to be the Banach space 
\[
T_{d}^{0,\alpha}(\Gamma)=\{\bJ\in T^{0,\alpha}(\Gamma):  {\rm Div }\bJ \in C^{0,\alpha}(\Gamma) \},
\]
where $\rm Div$ denotes the surface divergence on $\Gamma$. 
\begin{lemma}
 For the single layer potential $\phi$ in equation \eqref{singler} with  density $\sigma\in C^{0,\alpha}(\Gamma)$, the first derivative can be uniformly H\"older continuously  extended from $\mathbb{R}^3\setminus \overline{\Omega}$ to $\mathbb{R}^3\setminus {\Omega}$ with boundary value
	\[
	\nabla \phi_{+}(\bx) =  \int_\Gamma \nabla_{\bx} G^\kappa(\bx,\by)\sigma(\by)ds_\by-\frac{1}{2}\sigma(\bx)\bn(\bx), \quad \bx\in\Gamma  
	\] 
	where $$\nabla \phi_{+}:=\lim_{h\rightarrow +0 }\nabla \phi(\bx+ h\bn(x)).$$
\end{lemma}

Here we denote by $\nabla_\bx$ the gradient with respect to $\bx$. Proof can be found in \cite{DR-shu2}(Theorem 3.3)

\begin{lemma}
	 For the vector potential $\bpsi$ in equation \eqref{singlej} with  density $\bJ\in T^{0,\alpha}(\Gamma)$, the first derivative can be uniformly H\"older continuously  extended from $\mathbb{R}^3\setminus \overline{\Omega}$ to $\mathbb{R}^3\setminus {\Omega}$ with boundary value
	\[
	\nabla \times \bpsi_{+}(\bx) =  \int_\Gamma \nabla_{\bx} G^\kappa(\bx,\by)\times \bJ(\by)ds_\by-\frac{1}{2}\bn(\bx)\times\bJ(\bx), \quad \bx\in\Gamma  
	\] 
	where $$\nabla\times \bpsi_{+}:=\lim_{h\rightarrow +0 }\nabla \times \bpsi(\bx+ h\bn(x)).$$
\end{lemma}

Proof can be found in \cite{DR-shu2}(Theorem 6.12)

By making use of the two lemmas above, we are able to formulate the integral equations by letting $\bx$ approach the boundary  $\Gamma$ from the exterior, which yields the system
\begin{eqnarray}\label{elastcisys}
\begin{cases}
\left(-\frac{1}{2}+\mathcal{K}\right )\sigma(\bx) + \mathcal{N}\bJ(\bx) = f(\bx),  \\
\mathcal{H}\sigma(\bx) +\left(\frac{1}{2}+\mathcal{M}\right)\bJ(\bx) =  \boldsymbol{g}(\bx), 
\end{cases} \quad \bx\in\Gamma, 
\end{eqnarray}
where 
\begin{align}
&\mathcal{K}\sigma(\bx) = \int_{\Gamma}\bn(\bx)\cdot \nabla_{\bx} G^{\kp}(\bx,\by)
\sigma(\by)ds_{\by}, &&\mathcal{N}\bJ(\bx) = \int_{\Gamma}\bn(\bx)\cdot \nabla_{\bx} \times \big(G^{\ks}(\bx,\by)
\bJ(\by)\big)ds_{\by}, \notag \\
&\mathcal{H}\sigma(\bx) = \int_{\Gamma}\bn(\bx)\times \nabla_{\bx} G^{\kp}(\bx,\by)
\sigma(\by)ds_{\by}, &&\mathcal{M}\bJ(\bx) = \int_{\Gamma}\bn(\bx)\times \nabla_{\bx}\times \big(G^{\ks}(\bx,\by)
\bJ(\by)\big)ds_{\by}. \notag 
\end{align}

While all the four boundary integral operators in \eqref{elastcisys} are defined in the sense of Cauchy principle value, the operators $\mathcal{K}$ and $\mathcal{M}$ are only weakly singular for smooth surface $\Gamma$. Therefore they can be numerically discretized by the generalized Gaussian quadrature. Details will be given in  Section \ref{kerneleval}. In order to discretize the operators $\mathcal{H}$ and $\mathcal{N}$ to high order, we need the following regularization properties.
\begin{theorem}
	\label{them1}
	Let $\bta$ and $\btb$ be the unit tangential vectors of a closed $C^2$ surface $\Gamma$ and $\bn = \bta\times \btb$. For $\sigma \in C^{1,\alpha}(\Gamma)$, $0<\alpha<1$ and $i = 1,2$,  it holds
	\begin{eqnarray}\label{regu1}
		\int_{\Gamma}\frac{\partial G^\kappa(\bx,\by)}{\partial \bti(\bx)}\sigma(\by)ds_\by &=&\int_\Gamma G^\kappa(\bx,\by)  {\rm Div}_\by(\bn(\by)\times \bti (\bx)\times \bn(\by)) \sigma(\by) ds_\by \notag \\
		&& + \int_\Gamma G^\kappa(\bx,\by)  (\bn(\by)\times \tau_i(\bx)\times \bn(\by)) \cdot {\rm Grad}_\by\sigma(\by) ds_\by \notag \\
		&&-\bti(\bx) \cdot\int_\Gamma \frac{\partial G^\kappa(\bx,\by)}{\partial \bn(\by)}  \bn(\by) \sigma(\by)ds_\by
	\end{eqnarray}
for $\bx\in\Gamma$, where ${\rm Div}_{\by}$ and ${\rm Grad}_{\by}$ are the surface divergence and gradient on $\Gamma$ with respect to $\by$, respectively. 
\end{theorem}

\begin{proof}
For $i=1,2$, based on the identity $\nabla_\bx G^\kappa(\bx,\by)=-\nabla_\by G^\kappa(\bx,\by)$, the operator 
\begin{eqnarray*}
	\int_{\Gamma}\frac{\partial G^\kappa(\bx,\by)}{\partial \bti(\bx)}\sigma(\by)ds_\by &=& -\int_\Gamma \bti \cdot \nabla_\by G^\kappa(\bx,\by) \sigma(\by)ds_\by \\
	&=& -\int_\Gamma\bti(\bx) \cdot\left({\rm Grad_\by}G^\kappa(\bx,\by)+\bn(\by)\frac{\partial G^\kappa(\bx,\by)}{\partial \bn(\by)}\right)\sigma(\by)ds_\by\\
	&=& -\int_\Gamma (\bn(\by)\times \bti(\bx)\times \bn(\by))\cdot {\rm Grad}_\by G^\kappa(\bx,\by) \sigma(\by)ds_\by \\ &&-\bti(\bx) \cdot\int_\Gamma \frac{\partial G^\kappa(\bx,\by)}{\partial \bn(\by)}  \bn(\by) \sigma(\by)ds_\by,
\end{eqnarray*}
The conclusion follows from the two identities: 
\begin{eqnarray}
\int_\Gamma \bJ\cdot {\rm Grad} \sigma ds_\by&=&-\int_\Gamma \sigma {\rm Div}\bJ ds_\by, \label{iden1} \\
{\rm Div}(\sigma \bJ) &=& {\rm Grad}\sigma \cdot \bJ + \sigma  {\rm Div}\bJ \label{iden2}  
\end{eqnarray}
for any function $\sigma \in C^{1,\alpha}(\Gamma)$ and tangential vector $\bJ:=J^1\bta+J^2\btb\in T_d^{0,\alpha}(\Gamma)$.	
\end{proof}

\begin{theorem}
	\label{them2}
	Let $\bta$ and $\btb$ be the unit tangential vectors of a closed $C^2$ surface $\Gamma$ and $\bn = \bta\times \btb$. For $\bn\times\bJ =-J^2\bta+J^1\btb\in T_{d}^{0,\alpha}(\Gamma)$, $0<\alpha<1$, and $i = 1,2$,  it holds
	\begin{eqnarray}\label{regu2}
		\bn(\bx)\cdot \nabla_\bx\times \int_{\Gamma}G^\kappa(\bx,\by)\bJ(\by)ds_\by	
		&=& -\int_\Gamma G^\kappa(\bx,\by){\rm Grad}_\by(\bn(\bx)\cdot \bn(\by))\cdot (\bn(\by)\times \bJ(\by))ds_\by\notag \\ &&-\int_\Gamma G^\kappa(\bx,\by)\bn(\bx)\cdot \bn(\by){\rm Div}_\by(\bn(\by)\times \bJ(\by))ds_\by \notag \\ 
		&&-\bn(\bx)\cdot \int_\Gamma \frac{\partial G^\kappa(\bx,\by)}{\partial \bn(\by)} \bn(\by)\times \bJ(\by)ds_\by, 
	\end{eqnarray}
for $\bx\in\Gamma$.
\end{theorem}	
\begin{proof}
	The operator
	\begin{eqnarray*}
		&& \bn(\bx)\cdot \nabla_\bx\times \int_{\Gamma}G^\kappa(\bx,\by)\bJ(\by)ds_\by \\ &=& \bn(\bx)\cdot \int_\Gamma \nabla_\bx G^\kappa(\bx,\by)\times \bJ(\by)ds_\by \\
		&=&  \int_\Gamma \nabla_\by G^\kappa(\bx,\by) \cdot \bn(\bx) \times \bJ(\by)ds_\by \\
		&=& \int_\Gamma \left({\rm Grad}_\by G^\kappa(\bx,\by)+\bn(\by)\frac{\partial G^\kappa(\bx,\by)}{\partial \bn(\by)}\right)\cdot \bn(\bx) \times \bJ(\by) ds_\by \\
		&=& -\int_\Gamma G^\kappa(\bx,\by){\rm Div}_\by\left(\bn(\by)\times (\bn(\bx)\times \bJ(\by))\times \bn(\by)\right) ds_\by 
		-\bn(\bx)\cdot \int_\Gamma \frac{\partial G^\kappa(\bx,\by)}{\partial \bn(\by)} \bn(\by)\times \bJ(\by)ds_\by 
		\end{eqnarray*}
	Here the last equality is obtained by making use of the identity \eqref{iden1}. Combining identity \eqref{iden2}, we further have
	\begin{eqnarray*}	
		&& {\rm Div}_\by\left(\bn(\by)\times (\bn(\bx)\times \bJ(\by))\times \bn(\by)\right) \\
		&=& {\rm Div}_\by\left( (\bn(\bx)\cdot \bn(\by)J^1(\by))\btb(\by)
		-(\bn(\bx)\cdot \bn(\by)J^2(\by))\bta(\by)\right) \\
		&=& {\rm Grad} _\by(\bn(\bx)\cdot \bn(\by))\cdot (-J^2\bta+J^1\btb) +\bn(\bx)\cdot \bn(\by){\rm Div}_\by(-J^2\bta+J^1\btb)
	\end{eqnarray*}
which completes the proof.
\end{proof}
Based on Theorems \ref{them1} and \ref{them2}, we can transform the Cauchy-type singular operators $\mathcal{H}$ and $\mathcal{N}$ into operators with weak singularities only, which greatly simplify the numerical computation. In the next section, we will make use of these results and discuss the Fourier expansions of these operators along the azimuthal direction of the axisymmetric boundary $\Gamma$.
\begin{remark}
	Although the original boundary value problem \eqref{HelmholtzDec} has a unique solution for $\kp>0$ and $\ks>0$, the integral formulation \eqref{elastcisys} based on the single layer potentials \eqref{singlerep} does not admit a unique solution when $\kp$ is the eigenvalue of the interior Helmholtz Dirichlet problem or $\ks$ is the eigenvalue of the interior Maxwell problem with zero tangential electric field on $\Gamma$. Detailed analysis for the two dimensional case can be found in ~\cite{LL2019}. 
\end{remark}
\begin{remark}
	Another way to represent the potential functions $\phi$ and $\bpsi$ is given by
	\begin{eqnarray}\label{doublerep}
	\phi =D^{\kp}\sigma, \quad
	\bpsi = \nabla\times S^{\ks}\mathbf{J}, \mbox{ for } \bx\in \mathbb{R}^3\setminus \overline{\Omega},
	\end{eqnarray}
	where
	\begin{eqnarray*}
		D^{\kp}\sigma(\bx) &=& \int_{\Gamma} \frac{\partial G^{\kp}(\bx,\by)}{\partial \bn(\by)}\sigma(\by)ds_{\by},\label{doubler}  
	\end{eqnarray*}
	which does not admit a unique solution either  when $\kp$ is the eigenvalue of the interior Helmholtz Neumann problem or $\ks$ is the eigenvalue of the interior Maxwell problem with zero tangential magnetic field on $\Gamma$.    
\end{remark}
\begin{remark}
	To eliminate the resonance frequency, one can apply the combined field technique to represent  $\phi$ and $\bpsi$ as 
	\begin{eqnarray}\label{combinefield}
	\begin{cases}
	\phi =\lp D^{\kp}+\mi\lambda_0 S^{\kp}\rp \sigma, \\
	\bpsi = \lp \nabla\times S^{\ks}+ \frac{\mi \lambda_0}{\ks^{2}}\nabla\times \nabla\times S^{\ks}\rp \mathbf{J}, 
	\end{cases}
	\mbox{ for } \bx\in \mathbb{R}^3\setminus \overline{\Omega},
	\end{eqnarray}
	for some $\lambda_0>0$. However, since in this work we mainly focus on the high order numerical method for the elastic scattering from axisymmetric objects, we simply adopt the single layer representation \eqref{singlerep} by assuming that  $\kp$ is not the eigenvalue of the interior Helmholtz Dirichlet problem and $\ks$ is not the eigenvalue of the interior Maxwell problem. Extension of the algorithm to the integral representation \eqref{doublerep} and the combined field representation \eqref{combinefield} will be straightforward but tedious. In particular, we will need the regularization properties of three dimensional hypersingular operators. Details will be explored in the future work.   
\end{remark}

\begin{remark}
	Although our numerical scheme focuses on the axisymmetric bodies, the integral formulation \eqref{elastcisys}, as well as the regularization results in Theorems \ref{them1} and \ref{them2}, are applicable to general geometries with closed boundaries.
\end{remark}

\section{Fourier representation of the
  boundary integral operators}
\label{sec_fourier}
It is an area of active research to discretize integral equations in complex geometries in three dimensions to high order, which is highly non-trivial and rather computationally expensive~\cite{Bremer2012}. Some recent progress in this area can be found in \cite{GORV}. However, there exist many
important applications of elastic scattering from
axisymmetric objects (for instance, submarine detection, composite materials, etc).
In this case, variables can be separated in
cylindrical coordinates resulting in a system of decoupled line
integrals. The discretization and solution of integral equations along
curves in two dimensions is a much easier problem, and there exist many efficient
schemes~\cite{alpert1999, Brem2012, hao_2014, helsing2015}. In addition, the
resulting Fourier decomposition scheme easily parallelizes and can
address a range of rather complicated axisymmetric geometries. Discussion and details regarding the discretization of scalar-valued
integral equations along bodies of revolution are contained
in~\cite{helsing2014, YHM2012}. The vector-valued case, in particular integral equation
methods for Maxwell's equations, are discussed in~\cite{HK15, LAI2019}. In this section, we follow the discussion and notations in~\cite{LAI20171, LAI2019} to Fourier expand the operators in~\eqref{elastcisys}.

Consider an axisymmetric object $\Omega$ with rotational symmetry with
respect to the $z$-axis.  In other words, its
boundary~$\Gamma$ is generated by rotating curve $\gamma$ in
the $xz$-plane about the $z$-axis.
The goal is to design an efficient numerical
solver for the time-harmonic Navier
equations with the rigid boundary conditions along the surface of axisymmetric
objects, given in terms of the boundary integral system~\eqref{elastcisys}. It is convenient to consider the problem in cylindrical coordinates, which will be given as $(r,\theta,z)$, and we denote the
standard unit vectors by~$(\er,\et,\ez)$.  We assume that
the generating curve~$\gamma$ is parameterized counterclockwise
by~$(r(s),z(s))$, where~$s$ denotes arclength. This implies that the
unit tangential vector along the generating curve is~$\bt(s) = r'(s) \, \er +z'(s) \, \ez$, with $r'$ and $z'$
denoting differentiation with respect to arclength.  It is assumed that the generating curve $\gamma$ is smooth or contains a small number of corners. The unit tangential vectors on $\Gamma$ are given by $(\bt,\et)$. 
The unit exterior normal~$\bn$ is then given
by~$\bn(s) = z'(s) \, \er - r'(s) \, \ez$.
Since the cross section of ~$\Gamma$ in the
azimuthal direction is a circle, the density function $\sigma$ can be expanded in terms of the Fourier series with respect to~$\theta$,
\begin{equation}\label{expan1}
\sigma(r,\theta, z) = \sum_m   \sigma_m(r,z)e^{\mi m\theta}.
\end{equation}
 Similarly, the surface current~$\bJ$
on~$\Gamma$ can be written 
as~$\bJ = J^1 \, \bt + J^2 \, \et$. Taking the Fourier expansion of~$J^1$ and~$J^2$
with respect to~$\theta$ yields
\begin{equation}\label{expan2}
\bJ(r,\theta, z) = \sum_m  \lp J_m^1(r,z) \, \bt + J_m^2(r,z) \, \et \rp e^{\mi m\theta}.
\end{equation}
In the following, the dependence of the unit vectors on the
variables~$r,\theta,z$ will be omitted unless needed for clarity.

For numerical purpose, we need to investigate the Fourier expansions of the singular operators $\mathcal{K}$, $\mathcal{N}$, $\mathcal{H}$ and $\mathcal{M}$ in terms of the Fourier coefficients of $\sigma$ and $\bJ$ given in \eqref{expan1} and \eqref{expan2}. Let $(r_t, \theta_t, z_t)=(r(t), \theta_t,z(t))$ be a target point on $\Gamma$. Define the modal Green's functions as
\begin{align}
g_m^{\kappa,1}(r_t,z_t,r,z)& =\int_{0}^{2\pi}\frac{e^{\mi \kappa\rho}}{4\pi \rho}
e^{-\mi m\phi} \, d\phi, \label{green11}\\
g_m^{\kappa,2}(r_t,z_t,r,z)& =\int_0^{2\pi}\frac{e^{\mi \kappa\rho}}{4\pi \rho}
\cos{m\phi} \, \cos{\phi} \, d\phi, \label{green21}\\
g_m^{\kappa,3}(r_t,z_t,r,z)& = \int_{0}^{2\pi}\frac{e^{\mi \kappa\rho}}{4\pi
	\rho}
\sin{m\phi} \, \sin{\phi} \, d\phi, \label{green31}
\end{align}
where
\begin{equation}\label{eq:rho}
\rho =
\sqrt{r_t^2+r^2-2r_tr\cos{\phi}+(z_t-z)^2},
\end{equation}
 $\phi = \theta_t - \theta$ and $\kappa=\kp$  or $\ks$.

We begin by the Fourier expansion of single layer potentials.
\begin{lemma}\label{lemmascalpot}
	The single layer potential $S^\kappa\sigma$ in cylindrical coordinate
	has the Fourier expansion 
	\begin{align}\label{fourierdecompz}
	S^\kappa\sigma(r_t,\theta_t,z_t)
	= \sum_m \beta_m(r_t,z_t) e^{im\theta_t}
	\end{align}
	where 
	\begin{eqnarray}
	\beta^\kappa_m(r_t,z_t) = \int_{\gamma}  \sigma_m(s) \, r(s) \,
	g_m^{\kappa,1}(r_t,z_t,r(s),z(s)) \, ds. \label{equc0}
	\end{eqnarray}
\end{lemma}
 It is obtained by direct computation, which is also given in \cite{YHM2012,helsing2014}. 
From Lemma \ref{lemmascalpot}, the Fourier expansion for the boundary operator $\mathcal{K}\sigma$ with respect to the azimuthal direction is given as
\begin{eqnarray}
 \mathcal{K}\sigma(r_t,\theta_t,z_t) = \sum_m \lp z'_t\frac{\partial \beta^{\kp}_m}{\partial r_t}-r'_t\frac{\partial \beta^{\kp}_m}{\partial z_t}\rp e^{\mi m\theta_t}.
\end{eqnarray}
where we denote the
gradient of~$\gamma$ at~$(r_t,z_t)$ as $\nabla\gamma(t) = (r'_t,z'_t)$.

\begin{lemma}\label{lemmavecpot}
  The vector potential $S^\kappa\bJ$ in cylindrical coordinate
  has the Fourier expansion 
\begin{align}\label{fourierdecompj}
  S^\kappa\bJ(r_t,\theta_t,z_t)
  = \sum_m\lp c_m^{\kappa,1}(r_t,z_t) \, \er + c_m^{\kappa,2}(r_t,z_t) \, \et
  +c_m^{\kappa,3}(r_t,z_t) \, \ez \rp e^{im\theta_t}
\end{align}
where 
\begin{align*}
  c_m^{\kappa,1}(r_t,z_t) &=  \int_{\gamma} J^1_m(s) \, r(s) \, r'(s) \,
                    g^{\kappa,2}_m(r_t,z_t,r(s),z(s)) \, ds  \notag \\
                    &-\mi\int_{\gamma}J^2_m(s) \, r(s) \,
                    g_m^{\kappa,3}(r_t,z_t,r(s),z(s)) \, ds,  \\
  c_m^{\kappa,2}(r_t,z_t)  &=   \mi\int_{\gamma} J^1_m(s) \, r(s) \, r'(s) \,
                     g_m^{\kappa,3}(r_t,z_t,r(s),z(s)) \, ds \notag \\
                    &+\int_{\gamma} J^2_m(s) \, r(s) \,
                     g_m^{\kappa,2}(r_t,z_t,r(s),z(s))
                     \, ds, \label{equc2} \\
  c_m^{\kappa,3}(r_t,z_t) &= \int_{\gamma}  J^1_m(s) \, r(s) \, z'(s) \,
                    g_m^{\kappa,1}(r_t,z_t,r(s),z(s)) \, ds. 
\end{align*}
\end{lemma}

Using Lemma~\ref{lemmavecpot}, we can obtain the azimuthal
Fourier decomposition of the boundary operator ~$\mathcal{M}\bJ$:
\begin{multline}
 \mathcal{M}\bJ(r_t,\theta_t,z_t)
 = \\ \sum_{m}
 \lp
 \lp \frac{\partial c^{\ks,1}_m}{\partial z_t}-\frac{\partial c^{\ks,3}_m}{\partial
   r_t} \rp \bt 
   -\lp
 \frac{z_t'}{r_t}
 \lp c^{\ks,2}_m+r_t\frac{\partial c^{\ks,2}_m}{\partial r_t} -\mi mc^{\ks,1}_m \rp
+r_t' \lp \frac{\mi m}{r_t}c_m^{\ks,3}-\frac{\partial c^{\ks,2}_m}{\partial
  z_t} \rp\rp\et
\rp e^{\mi m\theta_t}.
\end{multline}
In order to derive the Fourier representations of boundary operators $\mathcal{H}$ and $\mathcal{N}$, by Theorems \ref{them1} and \ref{them2}, we need the Fourier expansions of the double layer boundary operators
\begin{eqnarray}
\mathcal{D}^\kappa(\sigma\bn )(\bx) &=& \int_\Gamma \frac{\partial G^\kappa(\bx,\by)}{\partial \bn(\by)}  \bn(\by) \sigma(\by)ds_\by \\
\mathcal{D}^\kappa(\bn\times\bJ)(\bx) & =& \int_\Gamma \frac{\partial G^\kappa(\bx,\by)}{\partial \bn(\by)}  \bn(\by) \times \bJ(\by)ds_\by
\end{eqnarray}
for $\bx\in \Gamma$, $\kappa=\kp$ or $\ks$, $\sigma\in C^{1,\alpha}(\Gamma)$ and $\bJ\in T^{0,\alpha}(\Gamma)$. For brevity, we drop the dependence on $(r_t,z_t,r(s),z(s))$ for all the derivatives of $g^{\kappa,i}_m$, $i=1,2,3$.

\begin{lemma}
	The double layer boundary operator $\mathcal{D}^\kappa(\sigma\bn )$ in cylindrical coordinate
	has the Fourier expansion  
	\begin{align}\label{fourierdecomp2}
	\mathcal{D}^\kappa(\sigma\bn )(r_t,\theta_t,z_t)
	= \sum_m\left( a_m^{\kappa,1}(r_t,z_t) \, \er + a_m^{\kappa,2}(r_t,z_t) \, \et
	+a_m^{\kappa,3}(r_t,z_t) \, \ez \right) e^{\mi m\theta_t}
	\end{align}
	where 
	\begin{align*}
	a_m^{\kappa,1}(r_t,z_t) &= \int_{\gamma} \sigma_m(s) \, r(s) \, z'(s) \,
	 \lp z'(s)\frac{\partial g^{\kappa,2}_m}{\partial r}-r'(s)\frac{\partial g^{\kappa,2}_m}{\partial z} \, \rp ds  \\
	a_m^{\kappa,2}(r_t,z_t) &= \mi\int_{\gamma} \sigma_m(s) \, r(s) \, z'(s) \,
	\lp z'(s) \frac{\partial g^{\kappa,3}_m}{\partial r}-r'(s)\frac{\partial g^{\kappa,3}_m}{\partial z} \, \rp ds  \\
	a_m^{\kappa,3}(r_t,z_t) &= -\int_{\gamma}  \sigma_m(s) \,  r'(s)r(s) \,
	\lp z'(s) \frac{\partial g^{\kappa,1}_m}{\partial r}-r'(s)\frac{\partial g^{\kappa,1}_m}{\partial z} \, \rp ds 
	\end{align*}
	
\end{lemma}

\begin{lemma}
	The double layer boundary operator $\mathcal{D}^\kappa(\bn\times \bJ )$ in cylindrical coordinate has the Fourier expansion 
	\begin{align}\label{fourierdecomp3}
	\mathcal{D}^\kappa(\bn\times \bJ )(r_t,\theta_t,z_t)
	= \sum_m\left( b_m^{\kappa,1}(r_t,z_t) \, \er + b_m^{\kappa,2}(r_t,z_t) \, \et
	+b_m^{\kappa,3}(r_t,z_t) \, \ez \right) e^{im\theta_t}
	\end{align}
	where 
	\begin{eqnarray*}
	b_m^{\kappa,1}(r_t,z_t) &=&  \int_{\gamma} J^2_m(s) \, r(s) \, r'(s) \,
	\lp z'(s)\frac{\partial g^{\kappa,2}_m}{\partial r} -r'(s)\frac{\partial g^{\kappa,2}_m}{\partial z}\rp \, ds  \notag \\ 
	&&+\mi\int_{\gamma}J^1_m(s) \, r(s) \,
	\lp z'(s)\frac{\partial g^{\kappa,3}_m}{\partial r} -r'(s)\frac{\partial g^{\kappa,3}_m}{\partial z}\rp \, ds,  \\
	b_m^{\kappa,2}(r_t,z_t) &=&  \mi\int_{\gamma} J^2_m(s) \, r(s) \, r'(s) \,
	\lp z'(s)\frac{\partial g^{\kappa,3}_m}{\partial r} -r'(s)\frac{\partial g^{\kappa,3}_m}{\partial z}\rp  \, ds \notag \\
	&&-\int_{\gamma} J^1_m(s) \, r(s) \,
	\lp z'(s)\frac{\partial g^{\kappa,2}_m}{\partial r} -r'(s)\frac{\partial g^{\kappa,2}_m}{\partial z}\rp 
	\, ds,  \\
	b_m^{\kappa,3}(r_t,z_t) &=& \int_{\gamma}  J^2_m(s) \, r(s) \, z'(s) \,
	\lp z'(s)\frac{\partial g^{\kappa,1}_m}{\partial r} -r'(s)\frac{\partial g^{\kappa,1}_m}{\partial z}\rp  \, ds. 
	\end{eqnarray*} 
\end{lemma}

For the other components in the right hand sides of formulae \eqref{regu1} and \eqref{regu2}, their Fourier representations can be found by combining Lemmas \ref{lemmascalpot},  \ref{lemmavecpot} and the Fourier expansions of surface gradient and divergence for axisymmetric shape $\Gamma$, which are given in the appendix. From these expansions, one can see that in order to evaluate the boundary operators ~$\mathcal{K}$, ~$\mathcal{N}$, ~$\mathcal{H}$ and~$\mathcal{M}$
rapidly, the values of~$a^{\kappa,1}_m$,~$b^{\kappa,2}_m$ and~$c^{\kappa,3}_m$, $i=1,2,3$, as well as the derivatives of $c^{\kappa,3}_m$, need to be computed efficiently.  The evaluation of these coefficients can be performed in two steps: (1) Compute
$g^{\kappa,i}_m$ and its derivatives, and then (2) integrate~$g_m^{\kappa,i}$ and its
derivatives (according to the above formulae) along the
generating curve~$\gamma$. Unfortunately, there are no numerically useful closed-form
expressions for~$g_m^{\kappa,i}$. Expansions of these functions in terms of half-order Hankel functions are  expensive to
compute~\cite{conway_cohl} and designing robust contour integration
methods for large values of~$m$ is quite
complicated~\cite{gustafsson2010}.  Furthermore, 
since~$g_m^{\kappa,i}$ has a logarithmic singularity when $(r_t,z_t)=(r(s),z(s))$, special care must be taken during the integration along~$\gamma$. Direct evaluation based on the adaptive integration has been proven to be time consuming and will be the bottleneck of the resulting solver~\cite{LAI20171}. Therefore, an efficient numerical scheme is needed to accelerate the two steps. Here we adopt an FFT based algorithm for the evaluation of $g_m^{\kappa,i}$ (as well as their derivatives), and apply the generalized Gaussian quadrature for the singular integration along $\gamma$. To avoid redundancy of the existing work, we only introduce some general ideas in the next section and refer interested readers to~\cite{helsing2014, LAI2019, YHM2012} for more details.

\section{Fast kernel evaluation and generalized Gaussian quadrature}
\label{kerneleval}
\subsection{Fast kernel evaluation}
To evaluate the modal Green's functions~\eqref{green11}-\eqref{green31}, including their derivatives,  we adopt the method based on FFT, recurrence relations and kernel splitting, as discussed in~\cite{epstein_2018,helsing2015,helsing2014}, to accelerate the computation of these kernel functions. 

Based on the fact that~$\rho$ in equation~\eqref{eq:rho} is
an even function with respect to~$\phi$ on~$[0,2\pi]$, we observe that
\begin{equation}
  \label{equrelation}
    g_m^{\kappa,2} = \frac{g_{m+1}^{\kappa,1}  + g_{m-1}^{\kappa,1}}{2}, \qquad 
    g_m^{\kappa,3} = -\frac{g_{m+1}^{\kappa,1} -  g_{m-1}^{\kappa,1}}{2}
\end{equation}
for any mode $m$. In other words,
we need only to evaluate~$g_m^{\kappa,1}$ and its derivatives efficiently, as the
other two kernel functions can be obtained by linear combinations \eqref{equrelation}. Therefore the goal is to efficiently evaluate all~$g^{\kappa,1}_m$ for $|m|\leq M$, where~$M>0$ is the bandwidth of the data in the azimuthal direction. In general, unless the wave numbers $\kp$ and $\ks$ are particularly high, only a modest number of
Fourier modes $m$ are needed for high-precision discretizations of
the integral equations. Once the incoming data has been Fourier transformed along the
azimuthal direction on the boundary, the number of Fourier modes $M$
needed in the discretization can be determined based on the decay of the
coefficients of the incoming field. 

The evaluation of $g^{\kappa,1}_m$ for $|m|\leq M$ can be split into two cases. To ease the discussion, we let $(r(s),z(s))$ be replaced by~$(r_s,z_s)$ for simplicity. 
\begin{enumerate}
	\item When the target~$(r_t,z_t)$ is far away from the source~$(r_s,z_s)$,
	the integral in~\eqref{green11} can be discretized using the periodic
	trapezoidal rule with $2M+1$-points and therefore the fast Fourier
	transform (FFT) can be used to evaluate all the ~$g^{\kappa,1}_m$ for $m=-M,-M+1,\ldots,M$.
	
	\item When~$(r_t,z_t)$ is near~$(r_s,z_s)$, the integrand is nearly
	singular and a prohibitively large number of discretization points
	would be needed to obtain sufficient quadrature accuracy.
	To overcome this difficulty, we
	apply the kernel splitting technique.
	The main idea is to explicitly split the
	integrand into smooth and singular parts. Fourier coefficients of the
	smooths parts can be obtained numerically via the FFT. The coefficients of the singular part can be obtained analytically via recurrence relations. The Fourier coefficients of the original kernel can then be obtained via discrete convolution. It has been successfully applied in~\cite{epstein_2018, helsing2014, LAI2019, YHM2012}, and therefore we skip the details that are given in the reference mentioned. In particular,~\cite{epstein_2018} provides estimates on the size of the FFT needed and other important tuning parameters.
\end{enumerate}
 
Finally, it is easy to show that in both cases, the computational complexity is $\mathcal{O}(M\log M)$, which implies the scheme is highly efficient. 

\subsection{Generalized Gaussian quadrature}
Once a scheme to evaluate the modal Green's functions and
their derivatives is in place, the next step is to discretize each decoupled modal
integral equation along the generating curve~$\gamma$.  Here we use a
Nystr\"om-like method for discretizing the integral equations.  Since
the modal Green's functions have logarithmic
singularities~\cite{cohl_1999, conway_cohl}, any efficient
Nystr\"om-like scheme will require a quadrature that accurately
evaluates weakly singular integral operators.  For high order
integration, one can construct the quadrature based on the kernel
splitting technique as in~\cite{Kress2010}. However, this will become
tedious given the various order of singularities in the derivatives
of~$g_m^{\kappa,i}$. For our numerical simulations, we implement a
panel-based discretization scheme using generalized Gaussian
quadratures. An in-depth discussion of generalized Gaussian
quadrature schemes is given in~\cite{BGR2010, hao_2014}. The panel-based discretization scheme of this paper, as opposed to that based on hybrid-Gauss trapezoidal rules~\cite{alpert1999, epstein_2018, oneil2018}, allows for adaptive discretizations.
Application of this panel-based discretization to the electromagnetic scattering from axisymmetric surfaces in three dimensions with edges and corners can be found in \cite{LAI2019}.

More specifically, given a kernel function~$g$ with logarithmic singularity and smooth
function~$\sigma$, our goal is to evaluate, with high-order accuracy,
the integral
\begin{equation}\label{singint}
\mathcal S \sigma(\bx) = \int_{\gamma} g(\bx,\by) \, \sigma(\by) \, ds_\by .
\end{equation}
Assume that the generating curve~$\gamma$ is smooth and parametrized
as~$\gamma(s)$, where $s$ is the arclength. The total arclength will be
denoted by $L$. We first divide~$[0,L]$ into~$N$ panels. The division
can be either uniform or nonuniform, depending on the particular geometry.
Each of the~$N$ panels, and therefore any function supported on it,
is discretized using $p$~scaled Gauss-Legendre nodes.
In a  Nystr\"om discretization scheme, the layer potential above
is approximated given as
\begin{equation}
  \mathcal S \sigma(\bx_{lm}) \approx
  \sum_{i = 1}^{N} \sum_{j=1}^{p} w_{lmij} \, g(\bx_{lm},\by_{ij}) \,
  \sigma(\vectorsym{y}_{ij}),
\end{equation}
where~$\bx_{lm}$ is the~$m$th Gauss-Legendre node on
panel~$l$,~$\by_{ij}$ is the~$j$th Gauss-Legendre node on panel~$i$,
and~$w_{lmij}$ is the Nystr\"om quadrature weight. For singular integrals,  it is generally difficult to have a uniform quadrature for all $\bx_{lm}$. We therefore split the integral \eqref{singint} into two parts and consider them separately, i.e.
\begin{equation}\label{singint2}
\mathcal S \sigma(\bx_{lm}) = \int_{\gamma\backslash \{\gamma_{l-1}\cup\gamma_{l}\cup \gamma_{l+1}\} } g(\bx_{lm},\by) \, \sigma(\by) \, ds_\by + \int_{\gamma_{l-1}\cup\gamma_{l}\cup \gamma_{l+1} } g(\bx_{lm},\by) \, \sigma(\by) \, ds_\by,
\end{equation}
where $\gamma_l$ denotes the $l$th panel. 

For integration on non-adjacent panels, namely, the first integral in the right hand side of \eqref{singint2}, we simply set
$w_{lmij}= w_{ij}$,
where~$w_{ij}$ is the standard scaled Gauss-Legendre weight on
panel~$i$.  Therefore, for non-adjacent panels, the order of accuracy
is expected to be~$2p-1$, although rigorous analysis and error estimates require knowledge of the regularity of the density function~$\sigma$.

For panels that are adjacent or self-interacting with $\gamma_{l}$, i.e., $\gamma_{l-1}\cup\gamma_{l}\cup \gamma_{l+1}$, it is difficult to derive efficient, high-order accurate Nystr\"om schemes in which the quadrature nodes are the same
as the discretization nodes (i.e. the points at which $\sigma$ are
sampled, the $p$ Gauss-Legendre nodes on each panel). Often it can be
very beneficial to use additional or at least different quadrature
\emph{support nodes} for approximating the integral. These support
nodes and corresponding weights in our case were computed using the
scheme of~\cite{BGR2010}. More specifically, to compute the integral over
panel~$\gamma_i$, for $i=l-1,\ l, \ l+1$, at a target~$\bx_{lm}$,
we approximate as
\begin{equation}\label{eq_nyslike}
  \begin{aligned}
     \int_{\gamma_{i} } g(\bx_{lm},\by) \, \sigma(\by) \, ds_\by\approx \sum_{j=1}^{p} c_{ij} \int_{\gamma_i} g(\bx_{lm},\by)
    \, P^i_j(\by)  \, ds_\by \approx \sum_{j=1}^{p}c_{ij} \, I_j(\bx_{lm}),
\end{aligned}
\end{equation}
where~$P^i_j$ is the scaled Legendre polynomial of degree~$j-1$ on
panel~$\gamma_i$ and~$c_{ij}$ are the Legendre expansion coefficients
of the degree~$j-1$ interpolating polynomial for~$\sigma$.  For a
fixed~$\bx_{lm}$, each $I_j(\bx_{lm})$ contains no unknowns and can
be evaluated via a pre-computed high-order generalized Gaussian
quadrature. The number of nodes required in these quadratures may vary.
Details on generating these quadratures were discussed
in~\cite{BGR2010}, and an analogous Nystr\"om-like discretization
scheme along surfaces in three dimensions was discussed
in~\cite{Bremer2012}. 

Furthermore, note that each $c_{ij}$ can be obtained via application
of a~$p\times p$ transform matrix acting on values of~$\sigma$. Denote
this transform matrix as~$\matrixsym{U}^{i}$, and it's entries
as~$U^{i}_{jn}$. Inserting this into~\eqref{eq_nyslike}, we have
\begin{equation}\label{eq_nyslike2}
  \begin{aligned}
    \int_{\gamma_{i} } g(\bx_{lm},\by) \, \sigma(\by) \, ds_\by\approx 
      \sum_{j=1}^{p}   \sum_{n=1}^p U^{i}_{jn} \, \sigma(\by_{in})
    \, I_j(\bx_{lm}) 
    =    \sum_{n=1}^p  \lp \sum_{j=1}^{p} U^{i}_{jn} \, I_j(\bx_{lm})
    \rp \sigma(\by_{in}).
\end{aligned}
\end{equation}
By simply exchanging the indexes $n$ and $j$, we obtain
\begin{equation}\label{eq_nyslike3}
\begin{aligned}
\int_{\gamma_{i} } g(\bx_{lm},\by) \, \sigma(\by) \, ds_\by\approx 
 \sum_{j=1}^p  \lp \sum_{n=1}^{p} U^{i}_{nj} \, I_n(\bx_{lm})
\rp \sigma(\by_{ij}),
\end{aligned}
\end{equation}
which implies we can take  $w_{lmij} \, g(\bx_{lm},\by_{ij})\approx \sum_{n=1}^{p} U^{i}_{nj} \, I_n(\bx_{lm}) $  to approximate the second part of the singular integral \eqref{singint2}.

In the case that~$\gamma$ is only piecewise smooth, a graded mesh near
the corner is used to maintain high-order accuracy. That is, after uniform
discretization of each smooth component of~$\gamma$, we perform a
dyadic refinement on panels that impinge on each corner point. Then
we apply the $p$th order generalized Gaussian quadrature on each of the
refined panels. This procedure, along with proper quadrature
weighting, has been shown to achieve high-order accuracy for curves with corners~\cite{Brem2012, HO08}.

\section{Numerical examples}
\label{sec_numeri}

In this section, we apply our algorithm to computing elastic
scattering from various elastically rigid axisymmetric objects.
In order to verify the
accuracy of our algorithm, we choose to test the extinction theorem by
solving the integral formulation~\eqref{elastcisys} with an
artificial solution. Specifically, we let the incident field in the exterior of ~$\Omega$ be generated by a polarized point source ~$\Phi(\bx,\by_0)$ with $\by_0$ located inside~$\Omega$, i.e.
\begin{equation}
  \begin{aligned}
    \boldsymbol{u}^{\rm inc}(\bx) = -\Phi(\bx,\by_0)\boldsymbol{p},
  \end{aligned}
\end{equation}
where
\begin{equation}
\Phi(\bx,\by_0)=\frac{1}{\mu}\lp G^{\ks}(\bx,\by_0)\boldsymbol{I} + \frac{1}{\ks^2}\nabla_{\bx}^T \nabla_{\bx} \lp G^{\ks}(\bx,\by_0)-G^{\kp}(\bx,\by_0)\rp \rp 
\end{equation}
is the the fundamental solution of the Navier equation~\cite{L2014}, $\boldsymbol{I}$ is the $3\times 3$ identity matrix, and~$\boldsymbol{p}$ is the polarization vector.  Since the incident field generated by the polarized point source satisfies the elastic equation in~$\mathbb{R}^3\setminus \overline{\Omega}$ with radiation condition, by extinction theorem, the scattered field is simply given by
\[
\usc(\bx) = \Phi(\bx,\by_0)\boldsymbol{p}.
\]
 Therefore, if the boundary condition on $\Gamma$ is specified by $\boldsymbol{u}^{\rm inc}$, by uniqueness, we should recover the field $\usc(\bx)$ by solving equation~\eqref{elastcisys}. Throughout all the numerical examples, we let
the source ~$\by_0$ be $(0.1,\, 0.1,\, 0.1)$ with
$\mu=1$ and $\boldsymbol{p}=(1,0,0)^T$.

In addition, we also solve the integral formulation~\eqref{elastcisys} with an incident
plane wave:
\begin{equation}
  \begin{aligned}
    \boldsymbol{u}^{\rm inc}(\bx)= (\boldsymbol{d}\cdot \boldsymbol{p})\boldsymbol{d}\mathrm{e}^{\mathrm{i} \kappa_{\mathfrak p}\boldsymbol{d}\cdot \bx}+(\boldsymbol{d}\times \boldsymbol{p})\times \boldsymbol{d} 
    \mathrm{e}^{\mathrm{i}\kappa_{\mathfrak s} \boldsymbol{d}\cdot \bx},
\end{aligned}
\end{equation}
where~$\vct{d}$ is the propagation direction and $\vct{d}\times \vct{p}$ is an
orthogonal polarization vector.
Throughout all the examples, we let
\begin{equation}
  \begin{aligned}
  \vct{d} &= \lp \cos\theta_1\sin\phi_1,\, \sin\theta_1\sin\phi_1,
  \, \cos\phi_1 \rp,\\
  \vct{p} &= \lp \cos\theta_2\sin\phi_2, \,  \sin\theta_2\sin\phi_2,
  \, \cos\phi_2 \rp, 
\end{aligned}
\end{equation}
with $\theta_1 = \pi/4$, $\phi_1 = \pi/8$, $\theta_2=\pi/5$, and
$\phi_2 = \pi/10$. The far field pattern of the scattered wave can be
found by letting~$|\bx|\rightarrow\infty$ in~\eqref{scatrep} and
using the asymptotic form of the Green's function~\eqref{greenf}:
\begin{eqnarray}\label{farfieldform}
  \boldsymbol{u}^{\infty}(\theta,\phi) = -\frac{1}{4\pi} 
   \lp \int_\Gamma \nabla_{\by}e^{-i\kp \, 
      {\hat{\bx}} \cdot \by } \sigma(\by) \, ds(\by)  
              + \int_\Gamma  \nabla_{\by}e^{-i\ks
              {\hat{\bx}} \cdot \by }\times \bJ(\by) \, ds(\by) \rp ,  
\end{eqnarray}
where~$\theta\in[0,2\pi]$ is the azimuthal angle,~$\phi\in[0,\pi]$
is the angle with respect to the positive~$z$ axis, and
${\hat{\bx}}=(\cos\theta\sin\phi,\sin\theta\sin\phi,\cos\phi)$
is a point on the unit sphere. 

We make use of the following notations in the subsequent tables that
present data from our scattering experiments:
\begin{itemize}
\item $\kp$: the compressional wavenumber,
\item $\ks$: the shear wavenumber,
\item $N_f$: the number of Fourier modes in the azimuthal direction
  used to resolve the solution. In other words, the Fourier modes are $-N_f,-N_f+1, \dots, N_f-1,N_f$.
\item $N_{pts}$: the total number of points used to
  discretize~$\gamma$,
\item $T_{matgen}$: the time (secs.) to construct the relevant matrix
  entries for all integral equations,
\item $T_{solve}$: the time  (secs.) to compute the
  $\mtx{LU}$-factorization of the system matrices for all modes,
\item $T_{syn}$: the time (secs.) to synthesize the density and surface current from their Fourier modes,
\item $E_{error}$: the absolute $L^2$ error of the elastic
  field measured at points placed on a sphere that encloses $\Omega$.
\end{itemize}  

The accuracy that controls the kernel evaluation (i.e. where to
truncate Fourier coefficients in the discrete convolutions) and the
number of Fourier modes in the decomposition of the incident wave is
set to be~$1\times 10^{-12}$. We apply the Nystr\"om-like discretization
described earlier in Section~\ref{kerneleval} to discretize the line integral
equations with $p=30$ on each panel.  In other words, we use $30$th-order
generalized Gaussian quadrature rules which contain~$30$ support nodes and
weights for self-interacting panels (which vary according to the
location of the target) and~$56$ support nodes and weights on adjacent
panels (which are target independent). All experiments were implemented
in~\textsc{Fortran 90} and carried out on a laptop with
four 2.0Ghz Intel cores and 16Gb of
RAM. We made use of~\textsc{OpenMP} for parallelism across decoupled
Fourier modes, and simple block $\mtx{LU}$-factorization
using~\textsc{LAPACK} for matrix inversion. Various fast direct
solvers such as \cite{GYM2012,HG2012,JL2014,Liu2015} could be applied
if larger problems were involved.  

Lastly, it is generally difficult to analytically parameterize the
generating curve with respect to arclength. In the following examples,
we have discretized the generating curve at panel-wise Gauss-Legendre points
in some parameter~$t$, not necessarily arclength. Any previous formulae which assumed arclength can be easily adjusted with factors of~$ds/dt$ to account for the parameterization metric. 

\begin{remark}
  The examples we tested here are all for the elastic scattering of a single
  object. The algorithm, however, can be extended to the scattering of
  multiple axisymmetric objects by transforming the incident field
  into the local coordinate of each object~\cite{GG2013,hao_2014}. In particular, we may use FMM developed for acoustic and Maxwell equations to accelerate the scattering among multi-particles, as presented in~\cite{LL2019} for the two dimensional case. This will be explored in our future investigation.
\end{remark}


\subsection*{Example 1: Scattering from a sphere}

Consider a sphere with the generating curve given by
\begin{equation}
  r(t) = 2\sin(t), \quad z(t) = 2\cos(t), \quad t\in[0,\pi].
\end{equation}
We solve the scattering problem of the sphere at several wavenumbers,
with accuracy results shown in Table~\ref{table_1}. By using at least 12
points per wavelength, we are able to achieve $8$ digits of accuracy
for most of the cases. Note that the CPU time is dominated by the
generation of matrix elements, which roughly scales
quadratically with the number of unknowns.   For the same number of unknowns, the computational time depends linearly on the number of
Fourier modes. Although the linear system for each mode is decoupled
from the other modes, which makes parallelization straightforward, the
computational time~$T_{solve}$ presented in Table~\ref{table_1} is the
total matrix inversion time via sequential solve.  Despite
its~$\mathcal{O}(N_fN_{pts}^3)$ complexity, it is still much smaller
than~$T_{matgen}$. Similar phenomenon has been observed in~\cite{LAI2019} for the  electromagnetic scattering from penetrable axisymmetric bodies. The time to synthesize the charge and current from their modes, as can be done by FFT, is negligible compared to the other parts during the solve. In addition, once the matrix is factorized, additional solves for new right hand sides are very fast.

For an incident plane wave with wavenumber~$\kp = 10$ and
$\ks=20$. Figure~\ref{figure_sphere}(A) shows the imaginary part of
the $x$ component of the scattered field $\usc$ in Cartesian coordinates, which is denoted by $\Im{\usc_1}$, at
the hemisphere with radius $4$. Figure~\ref{figure_sphere}(B) shows the imaginary part of the $x$ component of the far field pattern, denoted by $\Im\boldsymbol{u}^{\infty}_1$. Figure~\ref{figure_sphere}(C) plots the pointwise error for the point source scattering by comparing the numerical solution with the exact solution at the hemisphere. We can see that 9 digits accuracy is obtained for the elastic scattering from a sphere.

 \begin{figure}[h]
 	\centering
 	\begin{subfigure}[b]{.32\linewidth}
 		\centering
 		\includegraphics[width=.95\linewidth]{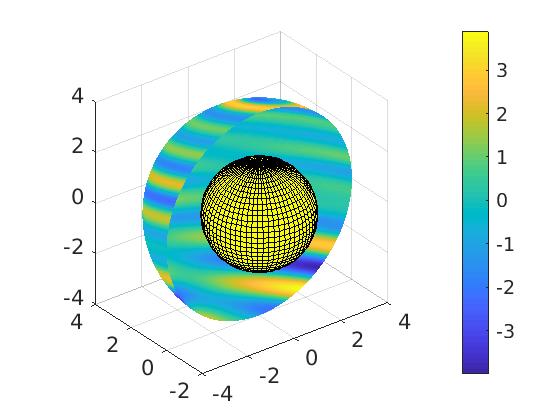}
 		\caption{$\Im{\usc_1}$.}
 	\end{subfigure}
 	\hfill
 	\begin{subfigure}[b]{.32\linewidth}
 		\centering
 		\includegraphics[width=.95\linewidth]{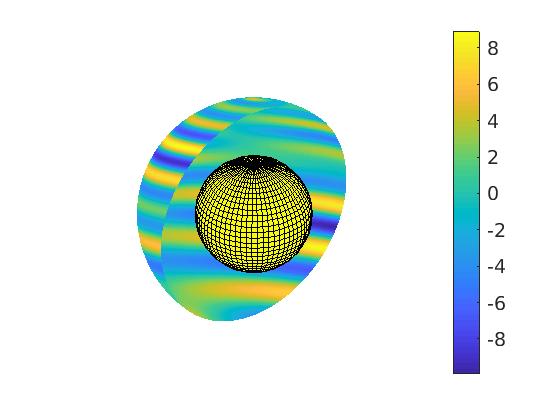}
 		\caption{$\Im\boldsymbol{u}^{\infty}_1$.}
 	\end{subfigure}
 	\hfill
 	\begin{subfigure}[b]{.32\linewidth}
 		\centering
 		\includegraphics[width=.95\linewidth]{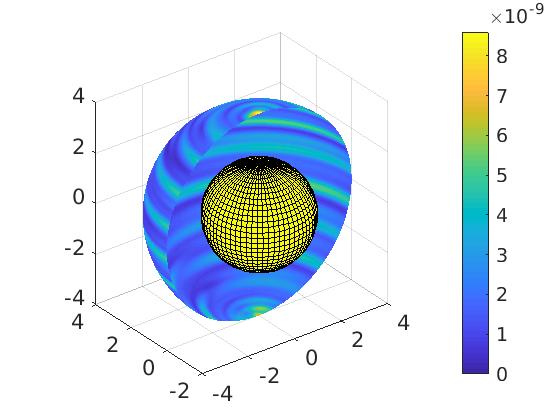}
 		\caption{Pointwise error.}
 	\end{subfigure}
 	\caption{Elastic scattering of a sphere with wavenumber~$\kp=10.0$ and
 		~$\ks=20.0$.}
 	\label{figure_sphere}
 \end{figure}
 
\begin{table}[t]
  \centering
  \caption{Results for the elastic scattering of a sphere at
    different wavenumber.}
  \label{table_1}
  \begin{tabular}{|c|c|c|c|c|c|c|c|}
    \hline
    $\kp$ &$\ks$ & $N_f$ & $N_{pts}$ & $T_{matgen}(s)$ & $T_{solve}(s)$ &$T_{syn}$ & $E_{error}$  \\
    \hline
    $1$ & $2$ & 12 & 120 & 2.33E+0 & 1.82E-1 & 6.55E-2 & 1.14E-11\\
    $1$ &$5$ & 12 &180 & 5.37E+0 & 4.18E-1 & 9.70E-2 & 8.51E-11 \\
    $1$ &$10$ & 15 &240 & 9.65E+0 & 9.46E-1  & 1.53E-1 & 3.82E-9 \\
    \hline
    $5$ &$2$ & 11 &180 & 5.38E+0 & 3.76E-1 & 8.68E-2 & 1.59E-10 \\
    $5$ &$10$ & 14 & 240 & 9.41E+0 & 9.01E-1 & 1.45E-1 & 2.27E-10 \\
    $5$ &$20$ & 18 & 300 & 1.83E+1 & 1.98E+0 & 2.52E-1 & 3.99E-9 \\
    \hline
    $10$ &$5$ & 14 & 240 & 9.96E+0 & 9.19E-1 & 1.48E-1 & 3.41E-10 \\
    $10$ &$20$ & 18 & 300 & 1.81E+1 & 1.97E+0 & 2.48E-1 & 4.96E-9 \\
    $10$ &$40$ & 24 & 480 & 5.88E+1 & 8.70E+0 & 5.77E-1 &1.84E-8\\
    \hline
    $20$ &$5$ & 17 &300 & 1.78E+1 & 1.89E+0 & 2.39E-1 & 1.26E-9 \\
    $20$ &$10$ & 17 & 300 & 1.81E+1 & 1.87E+0 & 2.37E-1 & 1.76E-9 \\
    $20$ &$40$ & 24 & 480 & 6.13E+1 & 8.63E+0 & 5.43E-1 & 1.37E-8 \\
    \hline
  \end{tabular}
\end{table}

 \subsection*{Example 2: Scattering from an ellipsoid}
 
 We next consider scattering from an ellipsoid whose generating
 curve is given by
 \begin{equation}
 r(t) = \sin(t), \quad z(t) = 2\cos(t), \quad t\in[0,\pi].
 \end{equation}

 
 The results are comparable with the sphere case. As we can see from Table \ref{table_2}, at small wavenumber, we
 obtain approximately~10 digits of accuracy. The accuracy
 slowly deteriorates to $7$ digits as the wavenumber increases.
 This is due to a stronger singularity near the poles (at $t=0$ and $t=1$) for
 higher wavenumbers.
 More digits can be obtained if additional refinement were implemented.
 Once again, if the system is factored, additional solves for new right-hand sides
 are inexpensive. This can be applied, for instance, to efficiently
 solving the elastic scattering problem with multiple incidences.
 
 For an incident plane wave with~$\kp=10.0$ and~$\ks=20.0$,
 Figure~\ref{figure_ellipsoid}(A) and (B) plot the imaginary parts of the $x$ component of the scattered elastic field and the far field pattern, respectively.  Figure~\ref{figure_ellipsoid}(C) shows the pointwise error for point source scattering by comparing with the exact solution. One can see that $8$ digits accuracy is obtained by roughly 20 points per wavelength.  
 
 \begin{figure}[h]
 	    \centering
 	    \begin{subfigure}[b]{.32\linewidth}
 	      \centering
 	      \includegraphics[width=.95\linewidth]{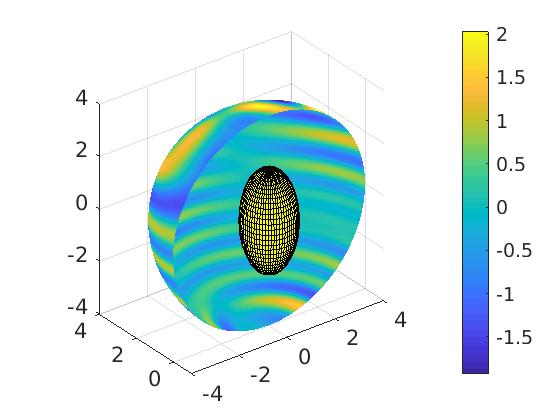}
 	      \caption{$\Im{\usc_1}$.}
 	    \end{subfigure}
 	    \hfill
 	    \begin{subfigure}[b]{.32\linewidth}
 	      \centering
 	      \includegraphics[width=.95\linewidth]{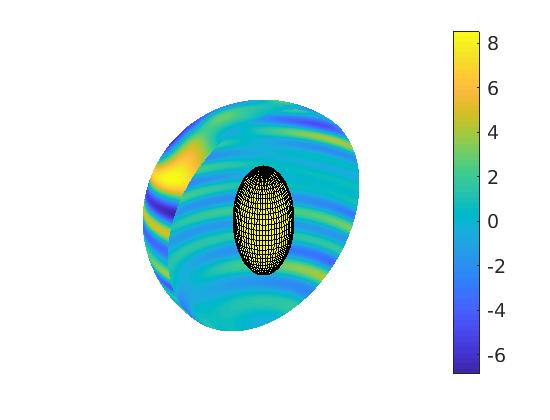}
 	      \caption{$\Im\boldsymbol{u}^{\infty}_1$.}
 	    \end{subfigure}
     \hfill
      	    \begin{subfigure}[b]{.32\linewidth}
     	\centering
     	\includegraphics[width=.95\linewidth]{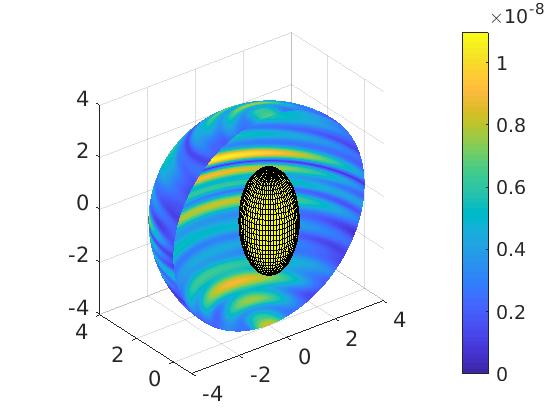}
     	\caption{Pointwise error.}
     \end{subfigure}
 	      \caption{Elastic scattering of an ellipsoid with wavenumber~$\kp=10.0$ and
 	      	~$\ks=20.0$.}
 	      \label{figure_ellipsoid}
 	  \end{figure}
 \begin{table}[t]
 	\centering
 	\caption{Results for the elastic scattering of an ellipsoid at
 		different wavenumber.}
 	\label{table_2}
 	\begin{tabular}{|c|c|c|c|c|c|c|c|}
 		\hline
 		$\kp$ &$\ks$ & $N_f$ & $N_{pts}$ & $T_{matgen}(s)$ & $T_{solve}(s)$ &$T_{syn}$ & $E_{error}$  \\
 		\hline
 		$1$ & $2$ & 16 & 120 & 2.35E+0 & 2.32E-1 & 9.23E-2 & 3.55E-12\\
 		$1$ &$5$ & 16 &180 & 5.60E+0 & 5.19E-1 & 1.44E-1 & 6.79E-11 \\
 		$1$ &$10$ & 16 &240 & 9.52E+0 & 1.07E+0  & 2.01E-1 & 4.58E-10 \\
 		\hline
 		$5$ &$2$ & 15 &180 & 5.55E+0 & 5.41E-1 & 1.15E-1 & 2.11E-10 \\
 		$5$ &$10$ & 16 & 240 & 9.56E+0 & 1.04E+0 & 1.83E-1 & 2.92E-10 \\
 		$5$ &$20$ & 19 & 300 & 1.69E+1 & 2.07E+0 & 2.59E-1 & 3.81E-9 \\
 		\hline
 		$10$ &$5$ & 15 & 240 & 9.38E+0 & 9.63E-1 & 1.53E-1 & 2.53E-9 \\
 		$10$ &$20$ & 18 & 300 & 1.65E+1 & 1.97E+0 & 2.50E-1 & 4.89E-9 \\
 		$10$ &$40$ & 24 & 480 & 5.40E+1 & 8.66E+0 & 5.77E-1 &9.79E-8\\
 		\hline
 		$20$ &$5$ & 18 &300 & 1.71E+1 & 2.09E+0 & 2.53E-1 & 2.45E-9 \\
 		$20$ &$10$ & 18 & 300 & 1.68E+1 & 1.98E+0 & 2.53E-1 & 1.75E-9 \\
 		$20$ &$40$ & 24 & 480 & 5.77E+1 & 8.65E+0 & 5.82E-1 & 7.37E-8 \\
 		\hline
 	\end{tabular}
 \end{table}
\subsection*{Example 3: Scattering from a rotated starfish}

For the third example, we consider an axisymmetric object whose
generating curve is given by
\begin{equation}
  \begin{aligned}
    r(t) &= [2+0.5\cos(5\pi (t-1))]\, \cos(\pi(t-0.5)),  \\
    z(t) &= [2+0.5\cos(5\pi (t-1))]\, \sin(\pi(t-0.5)),
  \end{aligned} 
\end{equation}
for $t\in [0,1]$. The generating curve is  open but gives rise to a smooth surface when crossing the~$z$-axis. We refer to this object as the rotated
\emph{starfish}, as shown in Figure~\ref{figure_wigg}.   We therefore
apply a uniform panel discretization to the parameter
space~$[0,1]$. Table~\ref{table_3} provides the accuracy results for
various wavenumbers and discretization refinements.  We obtain~$7$
to~$9$ digits of accuracy by using a sufficient number of
discretization points per wavelength. The computational time is again
dominated by the matrix generation,~$T_{matgen}$. As the matrix is
generated and factored, the time for additional solves is negligible.

Figure \ref{figure_wigg}(A) shows the scattering results with a
plane wave incidence for~$\kp=10$ and~$\ks = 20$. The far field
pattern is given in \ref{figure_wigg}(B). Pointwise error for point source scattering is shown in \ref{figure_wigg}(C). Roughly 7 digits accuracy is obtained for all the points on the hemisphere with radius 4. 
\begin{figure}[h]
	    \centering
	    \begin{subfigure}[b]{.32\linewidth}
	      \centering
	      \includegraphics[width=.95\linewidth]{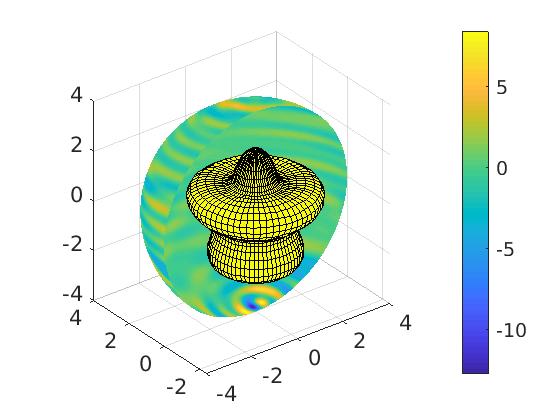}
	      \caption{$\Im{\usc_1}$.}
	    \end{subfigure}
	    \hfill
	    \begin{subfigure}[b]{.32\linewidth}
	      \centering
	      \includegraphics[width=.95\linewidth]{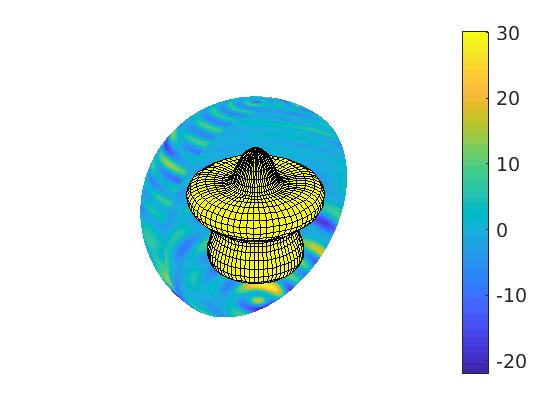}
	      \caption{$\Im\boldsymbol{u}^{\infty}_1$.}
	    \end{subfigure}
    \hfill
    \begin{subfigure}[b]{.32\linewidth}
    	\centering
    	\includegraphics[width=.95\linewidth]{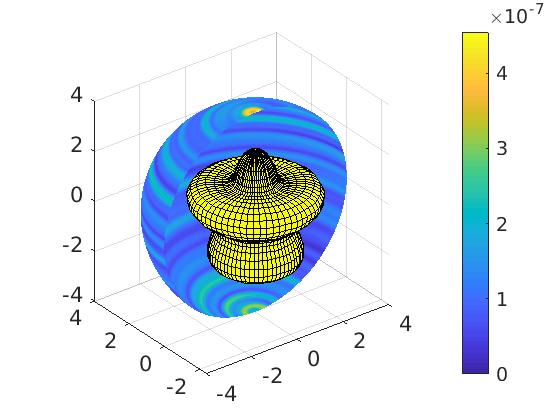}
    	\caption{Pointwise error.}
    \end{subfigure}
	    \caption{Elastic scattering of a rotated
	    	\emph{starfish} with wavenumber~$\kp=10.0$ and
	    	~$\ks=20.0$.}
	    \label{figure_wigg}
	  \end{figure}

\begin{table}[t]
	\centering
	\caption{Results for the elastic scattering of a rotated
		\emph{starfish} at
		different wavenumber.}
	\label{table_3}
	\begin{tabular}{|c|c|c|c|c|c|c|c|}
		\hline
		$\kp$ &$\ks$ & $N_f$ & $N_{pts}$ & $T_{matgen}(s)$ & $T_{solve}(s)$ &$T_{syn}$ & $E_{error}$  \\
		\hline
		$1$ & $2$ & 13 & 300 & 1.52E+1 & 1.46E+0 & 1.68E-1 & 2.06E-10\\
		$1$ &$5$ & 13 &360 & 2.13E+1 & 2.28E+0 & 2.05E-1 & 2.95E-10 \\
		$1$ &$10$ & 16 &420 & 3.31E+1 & 3.75E+0  & 2.63E-1 & 1.72E-8 \\
		\hline
		$5$ &$2$ & 12 &360 & 2.07E+1 & 2.11E+0 & 1.91E-1 & 1.66E-9 \\
		$5$ &$10$ & 15 & 420 & 3.35E+1 & 3.73E+0 & 2.63E-1 & 2.14E-8 \\
		$5$ &$20$ & 18 & 480 & 4.87E+1 & 6.34E+0 & 4.54E-1 & 1.63E-7 \\
		\hline
		$10$ &$5$ & 14 & 420 & 3.23E+1 & 3.47E+0 & 2.51E-1 & 1.14E-9 \\
		$10$ &$20$ & 18 & 480 & 4.88E+1 & 6.38E+0 & 4.53E-1 & 2.55E-7 \\
		$10$ &$40$ & 24 & 600 & 1.04E+2 & 1.63E+1 & 7.12E-1 &3.05E-7\\
		\hline
		$20$ &$5$ & 17 &480 & 4.72E+1 & 6.06E+0 & 4.44E-1 & 1.01E-7 \\
		$20$ &$10$ & 17 & 480 & 4.76E+1 & 6.06E+0 & 4.37E-1 & 1.25E-7 \\
		$20$ &$40$ & 24 & 600 & 1.04E+2 & 1.68E+1 & 7.09E-1 & 3.16E-7 \\
		\hline
	\end{tabular}
\end{table}

\subsection*{Example 4: Scattering from a droplet}
	In this example, we consider the scattering of a droplet with parametrization of the generating curve given by
	\begin{equation}
		\begin{aligned}
			    r(t) &= 4\sin(\pi t)\cos[0.5\pi(t-1.5)],  \\
			    z(t) &= 4\sin(\pi t)\sin[0.5\pi(t-1.5)]+2,
		\end{aligned}
	\end{equation} 	
	for $t\in[0.5,1]$. It has a point singularity at $t=1$, as shown by Figure \ref{figure_drop}. To resolve this singularity, we first apply a uniform panel discretization to the parameter space $[0.5,1]$. Then for the two panels that are adjacent to the singular point, we apply four times dyadic refinement, which yields a graded mesh. The scattering result for a point source at different wavenumber is given in Table~\ref{table_4}. The table shows more than 7 digits accuracy for the scattering at various wavenumber, which implies the point singularity of the droplet is fully resolved. 
	
	Similarly, we give the scattering result with a plane wave incidence for~$\kp=10$ and~$\ks = 20$ in Figure \ref{figure_drop}(A) and (B). Figure~\ref{figure_drop}(C) shows that 8 digits accuracy can be obtained for point source scattering. This is done in less than one minute computational time, which demonstrates the efficiency of our solver. 
	\begin{figure}[h]
	  	\centering
	  	\begin{subfigure}[b]{.32\linewidth}
	  		\centering
	  		\includegraphics[width=.95\linewidth]{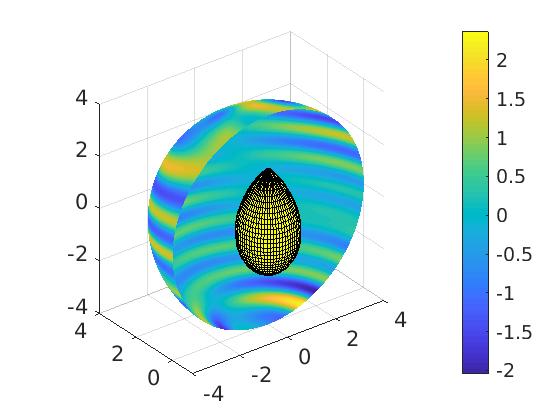}
	  		\caption{$\Im{\usc_1}$.}
	  	\end{subfigure}
	  	\hfill
	  	\begin{subfigure}[b]{.32\linewidth}
	  		\centering
	  		\includegraphics[width=.95\linewidth]{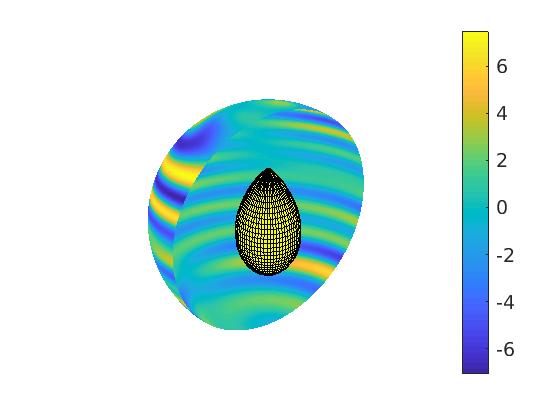}
	  		\caption{$\Im\boldsymbol{u}^{\infty}_1$.}
	  	\end{subfigure}
  	\hfill
  	\begin{subfigure}[b]{.32\linewidth}
  		\centering
  		\includegraphics[width=.95\linewidth]{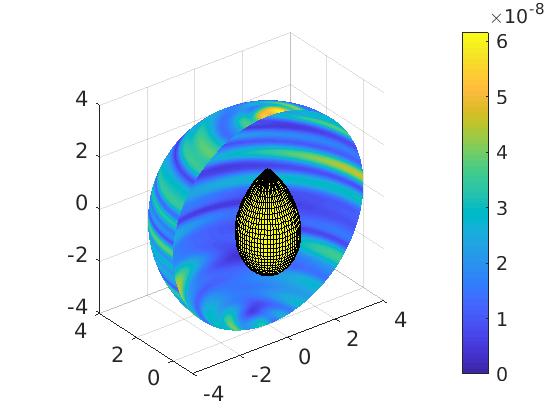}
  		\caption{Pointwise error.}
  	\end{subfigure}
	  	\caption{Elastic scattering of a droplet
	  		 with wavenumber~$\kp=10.0$ and
	  		~$\ks=20.0$.}
	  	\label{figure_drop}
	  \end{figure}

\begin{table}[t]
	\centering
	\caption{Results for the elastic scattering of a droplet at
		different wavenumber.}
	\label{table_4}
	\begin{tabular}{|c|c|c|c|c|c|c|c|}
		\hline
		$\kp$ &$\ks$ & $N_f$ & $N_{pts}$ & $T_{matgen}(s)$ & $T_{solve}(s)$ &$T_{syn}$ & $E_{error}$  \\
		\hline
		$1$ & $2$ & 16 & 300 & 1.66E+1 & 1.91E+0 & 2.40E-1 & 2.38E-10\\
		$1$ &$5$ & 15 &360 & 2.24E+1 & 2.56E+0 & 2.28E-1 & 3.14E-9 \\
		$1$ &$10$ & 16 &420 & 3.38E+1 & 4.02E+0  & 3.14E-1 & 3.26E-8 \\
		\hline
		$5$ &$2$ & 15 &360 & 2.25E+1 & 2.57E+0 & 2.30E-1 & 1.68E-9 \\
		$5$ &$10$ & 16 & 420 & 3.41E+1 & 3.98E+0 & 3.12E-1 & 2.57E-8 \\
		$5$ &$20$ & 19 & 480 & 4.78E+1 & 6.65E+0 & 4.63E-1 & 4.22E-8 \\
		\hline
		$10$ &$5$ & 15 & 420 & 3.33E+1 & 3.73E+0 & 2.69E-1 & 3.43E-9 \\
		$10$ &$20$ & 19 & 480 & 4.81E+1 & 6.69E+0 & 4.63E-1 & 3.42E-8 \\
		$10$ &$40$ & 24 & 600 & 9.82E+1 & 1.65E+1 & 6.98E-1 &9.11E-8\\
		\hline
		$20$ &$5$ & 18 &480 & 4.68E+1 & 6.33E+0 & 4.51E-1 & 1.38E-8 \\
		$20$ &$10$ & 18 & 480 & 4.69E+1 & 6.33E+0 & 4.49E-1 & 1.48E-8 \\
		$20$ &$40$ & 24 & 600 & 9.55E+1 & 1.66E+1 & 6.93E-1 & 3.61E-7 \\
		\hline
	\end{tabular}
\end{table}

\section{Conclusion}
\label{sec_con}
In this paper, we provided a derivation of the
boundary integral equation for the elastic scattering  from a three dimensional rigid obstacle in isotropic media based on Helmholtz decomposition.
Even though the resulting integral equation system is
not second-kind (see~\cite{LL2019} for discussions in two dimensional case), it remains
relatively well-conditioned even when the boundary has a modest number of
edges or geometric singularities.  Our
numerical algorithm strongly takes advantage of the axisymmetric
geometry  by using a Fourier-based separation of variables in
the azimuthal angle to obtain a
sequence of decoupled integral equations on a cross section of the
geometry. Using FFTs, we are able to efficiently evaluate the modal Green's
functions and their derivatives.
High-order accurate convergence is observed when discretizing the
integral equations using generalized Gaussian quadratures and an
adaptive Nystr\"om-like method.
Numerical examples show that the
algorithm can efficiently and accurately  solve the scattering problem from various axisymmetric objects, even in the presence of corner singularities. Extension of this work to multi-particle scattering, as well as the applications to inverse elastic scattering are still under investigation. We will report them in the future.

\appendix
\section{Fourier expansion for the surface gradient and divergence}

	Let $\Gamma$ be the boundary of a three dimensional axisymmetric object with parametrization given by $$\lp r(s)\cos\theta,r(s)\sin\theta,z(s)\rp, $$ and $(\bta,\btb)$ be the unit tangential vectors on $\Gamma$. Here we do not require the parameter $s$ to be arclength of the generating curve of $\Gamma$. 
	\begin{lemma}\label{surfdiv}
	The surface gradient of a function $\sigma\in C^{1,\alpha}(\Gamma)$ with Fourier expansion \eqref{expan1} is given by
	\begin{eqnarray}
	{\rm Grad} \sigma =\sum_m \left(\frac{1}{\sqrt{r'^2+z'^2}}\frac{\partial\sigma_m}{\partial s} \bta + \frac{\mi m }{r} \sigma_m \btb \right)e^{\mi m\theta }.
	\end{eqnarray}
	The surface divergence of a tangential vector $\bJ=J^1\bta+J^2\btb\in T_d^{0,\alpha}(\Gamma)$ with Fourier expansion \eqref{expan2} is given by
	\begin{eqnarray}
	{\rm Div} \bJ=\sum_m \left(\frac{1}{r\sqrt{r'^2+z'^2}}\left(r'(s)J_m^1+r (J_m^1)'\right) + \frac{\mi m }{r} J_m^2\right)e^{\mi m\theta }. 
	\end{eqnarray}
\end{lemma}

We first give the expression for ${\rm Grad}_\by(\bn(\bx)\cdot \bn(\by))$ and ${\rm Div}_\by (\bn(\by)\times \boldsymbol{\tau}_i(\bx) \times \bn(\by))$, $i=1,2$. Since  $$\bn(\bx)= C_t\left(z'(t)(\cos \theta_t, \sin\theta_t,0)-r'(t)(0,0,1)\right),$$
$$\bn(\by)= C_s\left(z'(s)(\cos \theta, \sin\theta,0)-r'(s)(0,0,1)\right),$$
where $C_s = 1/\sqrt{r'(s)^2+z'(s)^2}$, $ C_t = 1/\sqrt{r'(t)^2+z'(t)^2}$, it holds
\begin{eqnarray*}
	&&{\rm Grad}_\by(\bn(\bx)\cdot \bn(\by)) \\
	&=& C_s \lp C_tC'_s(z'(t)z'(s)\cos(\theta_t-\theta)+r'(t)r'(s))
	+C_tC_s(z'(t)z''(s)\cos(\theta_t-\theta)+r'(t)r''(s))\rp\bta \\
	&+&1/r(s)C_tC_s z'(t)z'(s)\sin(\theta_t-\theta)\btb.
\end{eqnarray*}

Similarly, since
\begin{eqnarray}
\bta(\bx)&=& C_t\left(r'(t)(\cos \theta_t, \sin\theta_t,0)+z'(t)(0,0,1)\right), \quad \btb(\bx)= (-\sin \theta_t, \cos\theta_t,0),\notag \\ 
\bta(\by)&=& C_s\left(r'(s)(\cos \theta, \sin\theta,0)+z'(s)(0,0,1)\right), \quad \btb(\by)= (-\sin \theta, \cos\theta,0),\notag 
\end{eqnarray}
by Lemma \ref{surfdiv}, it holds
\begin{eqnarray*}
	&&{\rm Div}_\by (\bn(\by)\times \bta(\bx) \times \bn(\by)) \\
	&=&C_s/r(s) \big[ r'(s)C_tC_s\big(r'(t)r'(s)\cos(\theta_t-\theta)+z'(t)z'(s)\big) \\
	&+&r(s)C_tC'_s\big(r'(t)r'(s)\cos(\theta_t-\theta)+z'(t)z'(s)\big) \\ &+&r(s)C_tC_s\big(r'(t)r''(s)\cos(\theta_t-\theta)+z'(t)z''(s)\big)\big] \\
	&-&1/r(s)C_tr'(t)\cos(\theta_t-\theta),
\end{eqnarray*}
 and
\begin{eqnarray*}
	&&{\rm Div}_\by (\bn(\by)\times \btb(x) \times \bn(\by)) \\
	&=&C_s/ r(s) \left[-r'(s)C_s r'(s)\sin(\theta_t-\theta)-r(s)(C'_sr'(s)+C_sr''(s))\sin(\theta_t-\theta) \right]\\&+&1/r(s)\sin(\theta_t-\theta).
\end{eqnarray*}

To numerically construct the surface gradient operator for an unknown function $\sigma$, we make use of the fact that $\sigma$ is discretized on $p$ scaled Gauss-Legendre nodes on each panel as discussed in Section~\ref{kerneleval}.  Following the notation in Section~\ref{kerneleval}, $\sigma$ on the $i$th panel $\gamma_i$ is approximated by
$$\sigma(\by) \approx \sum_{j=1}^pc_{ij}P^i_j(\by)= \sum_{j=1}^p\sum_{n=1}^p U^{i}_{jn} \, \sigma(\by_{in}) P^i_j(\by). $$  
Therefore the surface divergence can be approximated as
$${\rm Grad}_\by\sigma(\by) \approx \sum_{j=1}^p \sum_{n=1}^p U^{i}_{jn} \, \sigma(\by_{in}) {\rm Grad}_\by P^i_j(\by)=\sum_{n=1}^p \sum_{j=1}^p U^{i}_{jn}{\rm Grad}_\by P^i_j(\by) \, \sigma(\by_{in}) .$$
By taking $\by$ to be the $p$ scaled Gauss-Legendre nodes on the $i$th panel, we obtain the discretized surface gradient operator for $\sigma$. The discretization of surface divergence operator for a vector function $\bJ$ can be similarly constructed. 

Combining the results above with Lemmas \ref{lemmascalpot} and \ref{lemmavecpot}, we  are able to obtain the azimuthal Fourier decomposition of boundary operators $\mathcal{H}$ and $\mathcal{N}$. Details are omitted.

\end{document}